\documentclass[11pt]{amsart}
\numberwithin{equation}{section}

\usepackage{amsmath,amssymb,amscd,amsxtra,latexsym,
graphicx,epsfig,euscript,amsthm,mathscinet}
\usepackage{amsmath,amsfonts,amsthm,amsopn,amssymb,enumitem}
\usepackage{graphicx}
\usepackage{float}
\usepackage{a4wide}
\usepackage{color}
\usepackage{esint}
\usepackage{palatino}
\usepackage[backend=bibtex,style=numeric,sorting=nyt,
giveninits=true,isbn=false,url=false,doi=false]{biblatex}
\renewbibmacro{in:}{\ifentrytype{article}{}{\printtext{\bibstring{in}\intitlepunct}}}
\DeclareFieldFormat*{title}{#1}
\usepackage[colorlinks]{hyperref}
\AtBeginDocument{%
  \hypersetup{%
    linkcolor=blue,%
    citecolor=green,%
  }%
}
\usepackage{cleveref}
\crefname{equation}{}{}
\crefname{figure}{}{}

\makeatletter
\@namedef{subjclassname@2020}{\textup{2020} Mathematics Subject Classification}
\makeatother

\newtheorem{theorem}{Theorem}[section]
\newtheorem{lemma}[theorem]{Lemma}

\newtheorem{proposition}[theorem]{Proposition}

\begin{document}

\title[Bifurcation of overdetermined capillary problems]{Bifurcation of overdetermined capillary problems in a strip domain}

\author{Yuanyuan Lian}
\address{(Y.~Lian)
  Departamento de Geometr\'{i}a y Topolog\'{i}a, Departamento de An\'alisis matem\'atico, Universidad de Granada,
  Campus Fuentenueva, 18071 Granada, Spain}
\email{lianyuanyuan.hthk@gmail.com; yuanyuanlian@correo.ugr.es}

\author{Pieralberto Sicbaldi}
\address{(P.~Sicbaldi)
  IMAG, Departamento de An\'alisis matem\'atico, Universidad de Granada,
  Campus Fuentenueva, 18071 Granada, Spain
  \& Aix Marseille Universit\'e - CNRS, Centrale Marseille - I2M, Marseille, France}
\email{pieralberto@ugr.es}
\thanks{{\bf Acknowledgements.}
Y. L. and P. S. have been supported by the Grants PID2020-117868GB-I00 and PID2023-150727NB-I00 of the MICIN/AEI.
P. S. has been supported also by the \emph{IMAG-Maria de Maeztu} Excellence Grant CEX2020-001105-M funded by the MICIN/AEI.}



\maketitle

\noindent

\noindent
\begin{abstract}
 In this paper, we consider the classical overdetermined capillary problem:
 \begin{equation*}
  \begin{cases}
  \mathrm{div} \left(\frac{\nabla u}{\sqrt{1+|\nabla u|^2}}\right) - bu =0 &~~\mbox{in}~~ \Omega,\\
  \partial_{\nu} u=\kappa &~~\mbox{on}~~\partial\Omega,\\
  u=c &~~\mbox{on}~~\partial\Omega,
  \end{cases}
\end{equation*}
 where $b$, $c$ and $\kappa$ are positive constants, and $\Omega\subset \mathbb{R}^2$. When $\Omega$ is an infinite strip, i.e., a domain bounded by two parallel straight lines, there exists a unique one-dimensional solution (called the trivial solution) to this problem. By means of a bifurcation argument, we establish the existence of a critical period $T_*$ at which a branch of non-trivial solutions bifurcates from the trivial one. These solutions are genuinely two-dimensional and are defined in unbounded periodic domains $\Omega$ that are diffeomorphic to an infinite strip, yet whose boundaries are no longer straight lines. This result offers a significant physical interpretation in the context of capillary phenomena.
\end{abstract}

\section{Introduction}
\label{section 1}

When a solid is inserted into a large liquid reservoir, a capillary phenomenon occurs at the interface between the solid and the liquid. Let $(x, t, z)$ denote the standard coordinates in the Euclidean space $\mathbb{R}^3$. The $x$- and $t$-axes lies in the horizontal plane (corresponding to the liquid level before insertion), while the $z$-axis is oriented in the opposite direction of gravity. Let $\Omega \subset \mathbb{R}^2$ denote the projection of the liquid surface onto the horizontal plane, and let $u$ represent the height of the liquid surface. Along the interface, the liquid surface meets the solid at a constant contact angle $\theta$, also referred to as the \emph{wetting angle}. The equilibrium configuration of the liquid surface is characterized by a height function $u$ whose mean curvature is proportional to the height itself. More precisely, $u$ satisfies the following problem, governed by the Laplace-Young law (see \cite[Chapter 1]{MR816345}, see also \cite{MR2335074}):
\begin{equation}\label{equil}
  \begin{cases}
  \mathrm{div} \left(\frac{\nabla u}{\sqrt{1+|\nabla u|^2}}\right) - b u = 0 & \mbox{in}~~ \Omega,\\
  \partial_{\nu} u = \cos \theta \sqrt{1+ |\nabla u|^2} & \mbox{on}~~ \partial \Omega,
\end{cases}
\end{equation}
where $b = (\rho-\rho_0) g/\sigma>0$ and $\nu$ denotes the unit outward normal on $\partial \Omega$. The parameters $\rho, \rho_0, g, \sigma$ represent the density of the liquid, the density of the surrounding gas, the gravitational acceleration and the surface tension of the liquid, respectively. The contact angle $\theta$ provides a quantitative measure of the wettability of a solid surface by a liquid, with its value dictating the behavior of the liquid-solid interface. It is generally broken down into three distinct regimes: high wetting ($0<\theta < \pi/2$), neutral wetting ($\theta = \pi/2$) and low wetting ($\pi/2 < \theta < \pi$). In the case $\theta = \pi/2$, there is neither capillary rise nor capillary depression. From the view point of mathematics, $u\equiv 0$ is the unique, trivial solution to \eqref{equil} for any domain. Due to the triviality of this neutral case, we assume $0<\theta <\pi/2$ throughout this paper. The case $\pi/2 < \theta < \pi$ can be reduced to the high wetting case under the transformation $u$ to $-u$.

If we further assume that the liquid rises a uniform height along the liquid-solid interface, then $u$ becomes a solution to the following overdetermined capillary problem:
\begin{equation}\label{capil}
  \begin{cases}
  \mathrm{div} \left(\frac{\nabla u}{\sqrt{1+|\nabla u|^2}}\right) - bu =0 &~~\mbox{in}~~ \Omega,\\
  \partial_{\nu} u=\kappa &~~\mbox{on}~~\partial\Omega,\\
  u=c &~~\mbox{on}~~\partial\Omega,
  \end{cases}
\end{equation}
where $c$ is a positive constant and $\kappa=\cot \theta >0$.

For the classical capillary phenomenon involving a liquid rising in a thin, straight cylindrical tube (where $\Omega$ is a bounded domain), the celebrated work of Serrin \cite[Theorem 2]{MR333220} establishes that the existence of a solution of \eqref{capil} implies that the tube's cross-section must be circular and the solution $u$ must be radially symmetric. In the case where a finite number of cylindrical solids, say $m$ solids with cross-section $G_i$ ($i=1,\dots, m$), are dipped into an infinitely large reservoir, the domain $\Omega=\mathbb{R}^2\backslash \cup_{i=1}^m G_i$ becomes an exterior domain. Based on the results by Reichel \cite{MR1463801} and Sirakov \cite{MR1808026}, it is known for such an overdetermined problem to be solvable, $m$ must equal $1$ and the cross-section $G_1$ must be a disk. In our recent work \cite{LianSicbaldi2025}, we investigate the configuration where the liquid in a large reservoir is interrupted by a single, infinitely long cylindrical plate. Here, the cross-section of the liquid domain $\Omega$ is an unbounded region that is diffeomorphic to the half-plane $\mathbb{R}^2_+$. We prove that under these conditions, the plate must be flat and the solution $u$ is necessarily one-dimensional.

\medskip

A natural and intriguing question arises: does a similar rigidity result hold for \emph{a pair of} infinitely long, non-intersecting cylindrical plates immersed in an infinitely large reservoir? In this paper, we demonstrate that this is not the case. Specifically, we show the existence of nontrivial unbounded domains that bifurcate from the trivial configuration of two parallel straight lines. In other words, the uniform capillary rise of a liquid along the entire boundary $\partial \Omega$ is not restricted to the configuration of two flat plates.

\medskip

Our approach relies on bifurcation analysis. The procedure begins by formulating a sufficiently smooth mapping, $F(x,\lambda)$, derived from the overdetermined boundary conditions. Here, $x$ represents the state variable in an appropriate Banach space and $\lambda \in \mathbb{R}$ denotes the bifurcation parameter. We first identify a family of trivial solutions, $F(0,\lambda)=0$, which exists for all $\lambda$. A nontrivial solution branch can only emerge at the points where $F$ loses its invertibility, i.e., where the Fr\'{e}chet derivative with respect to $x$, denoted by $F_{x}(0,\lambda)$, becomes singular. Such a bifurcation point occurs at a critical value $\lambda_*$ where the linearized operator $F_{x}(0,\lambda_*)$ has a zero eigenvalue. To establish the existence of a nontrivial branch, we employ the Crandall-Rabinowitz theorem (or the transversality condition). Specifically, we evaluate the mixed partial derivative  $F_{x\lambda}(0,\lambda_*)$; if the action of $F_{x\lambda}(0,\lambda_*)$ on the null-space eigenfunction lies strictly outside the range of $F_{x}(0,\lambda_*)$, the transversality condition is satisfied, ensuring the existence of a nontrivial bifurcation branch.

Bifurcation theory serves as a powerful tool for identifying the conditions under which symmetry or rigidity is broken. In their seminal work, Gidas, Ni and Nirenberg \cite{MR544879} proved that any positive solution to the equation $\Delta u +f(u)=0$ in a ball with homogeneous Dirichlet boundary conditions must be radially symmetric. In stark contrast to the case of a ball, Smoller and Wasserman \cite{MR853965, MR848648} demonstrated that symmetry breaking occurs on annuli. By employing bifurcation techniques, they showed that the equation exhibits a multiplicity of non-radial solutions alongside the expected radial ones.

A similar dichotomy exists in the study of constant mean curvature (CMC) surfaces. Alexandrov Theorem \cite{MR143162} establishes a rigidity result, stating that any embedded, compact CMC surface in three-dimensional Euclidean space must be a standard sphere. However, utilizing bifurcation theory, Mazzeo and Pacard \cite{MR1941630} proved that as the neck parameter of Delaunay nodoids approaches certain critical values, infinitely many bifurcation points arise, i.e., new families of CMC surfaces emerge.

In \cite{MR1470317}, Berestycki, Caffarelli and Nirenberg proposed the following celebrated conjecture regarding the overdetermined problem:
\begin{equation}\label{bcnn}
  \begin{cases}
\Delta u + f(u) = 0 & \mbox{in}~~ \Omega,\\
u> 0 & \mbox{in}~~ \Omega, \\
u= 0 & \mbox{on}~~ \partial \Omega, \\
\partial_{\nu} u=\kappa &\mbox{on }~~ \partial \Omega,
\end{cases}
\end{equation}
where $\kappa$ is a constant. The conjecture states that if $\mathbb{R}^n \backslash \overline{\Omega}$ is connected, then the existence of a bounded solution to \eqref{bcnn} implies that $\Omega$ must be a a ball, a half-space, a generalized cylinder $B^k \times \mathbb{R}^{n-k}$ (where $B^k$ is a $k$-dimensional ball), or the complement of one of these domains. However, this conjecture was disproved. Specifically, the second author provided the first counterexample by employing a bifurcation argument to establish the existence of non-standard domains admitting such solutions \cite{MR2592974}. Subsequently, numerous counterexamples have been constructed using bifurcation theory. These include domains topologically equivalent to a cylinder \cite{MR4484836, MR2854185, MR4649188}, a half-space \cite{MR3417183}, and the complement of a ball \cite{MR4046014, dai2024}.

\medskip

Inspired by \cite{MR2592974, MR2854185, MR4484836}, we anticipate the existence of nontrivial solutions to the overdetermined capillary problem \eqref{capil} in domains $\Omega$ that are diffeomorphic to an infinite strip. We first identify the known solutions corresponding to the capillary phenomenon between two parallel flat plates as the trivial branch, serving as the starting point for our bifurcation analysis. Following the methodology in \cite{MR2592974, MR2854185, MR4484836}, we impose a periodicity constraint on the boundary $\partial \Omega$ with a given period $T$. This effectively compactifies the infinite direction $\mathbb{R}$ into the circle $S^1$, ensuring that the linearized operator, which would otherwise have a continuous spectrum, possesses a purely discrete spectrum. This recovery of the Fredholm property allows us to analyze the degeneracy of the operator and apply the Crandall-Rabinowitz Theorem to establish the loss of rigidity, thereby confirming the emergence of a nontrivial bifurcation branch.

A key innovation of this paper is that, unlike previous studies centered on the standard Laplacian ($\Delta u$), our analysis tackles a more complex quasilinear operator associated with the mean curvature term
\begin{equation*}
  \mathrm{div} \left(\frac{\nabla u}{\sqrt{1+|\nabla u|^2}}\right).
\end{equation*}
This introduces significant technical challenges in both the linearization procedure and spectral analysis of the associated mapping. The presence of the mean curvature operator necessitates a more delicate treatment of the shape derivative, as the coefficients of the linearized operator now depend on the gradient of the trivial solution. In addition, not only is it technically more complicated, but it also requires us to come up with new method to study it.
To the best of our knowledge, there is a scarcity of literature \cite{MR4271788, LianSicbaldi2025} concerning the analogue of problem \eqref{bcnn} where the Laplacian is replaced by the mean curvature operator.

In addition, unlike traditional approaches that derive Neumann conditions from Dirichlet data, we move from the Neumann problem to the Dirichlet problem. This approach is dictated by the physics of capillarity: for any given liquid-solid interface, the contact angle is a fixed physical constant. Our study explores whether bifurcation occurs under these fixed-angle constraints. Specifically, we examine whether two cylindrical plates in a large reservoir can support multiple equilibrium solutions that satisfy the same height requirements.

Starting from the Neumann problem, we characterize the perturbed surface by its deformation $v$ relative to the equilibrium state. This allows us to define a smooth mapping $G(v, T): C^{2,\alpha}(\bar{\Omega}) \times \mathbb{R}_+ \to C^{2,\alpha}(\bar{\Omega})$. The linearized operator of $G$ at $v=0$, denoted by $H_T$, is composed of a term proportional to the identity plus a compact operator. Thus, $H_T$ is a Fredholm operator of index $0$, facilitating the application of bifurcation theory.

\medskip

Our main result is formulated as follows:
\begin{theorem} \label{Th_main_B_Cap}
For any $\kappa>0$, there exist a positive number $T_*$ and a continuous curve
\begin{equation*}
  \begin{array}{ccc}
  (-\varepsilon, \varepsilon) & \to & C^{2,\alpha}(\mathbb{R} / \mathbb{Z}) \times \mathbb{R}_+ \times \mathbb{R}_+\\
  \eta & \mapsto & (v_{\eta},T_{\eta},c_{\eta})
  \end{array}
\end{equation*}
for some $\varepsilon$ small, with
\begin{equation*}
  v_{\eta}(t)=v_{\eta}(-t),\quad \int_0^1 v_{\eta}(t)\,dt=0.
\end{equation*}
and $v_{\eta}=0$ if and only if $\eta=0$. Moreover, $T_0=T_*$ and
the following overdetermined capillary problem:
\begin{equation}\label{capil_result}
  \begin{cases}
  \mathrm{div} \left(\frac{\nabla u}{\sqrt{1+|\nabla u|^2}}\right) - bu =0 &~~\mbox{in}~~ \Omega_{\eta},\\
  \partial_{\nu} u=\kappa &~~\mbox{on}~~\partial\Omega_{\eta},\\
  u=c_{\eta} &~~\mbox{on}~~\partial\Omega_{\eta},
  \end{cases}
\end{equation}
has a solution in the domain
\begin{equation*}
  \Omega_{\eta} \,=\,
\left\{ (x,t) \in \mathbb{R} \times \mathbb{R}  \, :  \ |x| < 1 + v_{\eta} \left(\frac{t}{T_{\eta}}\right) \right\}.
\end{equation*}
The solution $u= u_{\eta}$ of problem \eqref{capil_result} is $T_{\eta}$-periodic in the variable $t$ and hence bounded.

In the context of the capillary phenomenon, if we insert two cylindrical plates, whose generating curves are given by
\begin{equation*}
x= 1 + v_{\eta} \left(\frac{t}{T_{\eta}}\right),~~x= -1 + v_{\eta} \left(\frac{t}{T_{\eta}}\right)
\end{equation*}
respectively, into an infinitely large liquid reservoir, the liquid between the plates rises to a uniform height along their inner walls.
\end{theorem}

%\begin{remark}\label{re.1}
%Unlike the Laplacian case where the linearized operator has constant coefficients, the linearization of the mean curvature operator involves weights that depend on the gradient of the trivial solution.
%\end{remark}

%\medskip

The remainder of this paper is organized as follows. In \Cref{sec:baseline solution}, we recall the existence of the solution to the capillary problem in an arbitrary domain and present the trivial solution for the capillary problem between two parallel flat plates. \Cref{sec:operator_G} is devoted to the definition and analysis of a nonlinear operator $G$ corresponding to the overdetermined problem \eqref{capil}. In \Cref{sec:linearized_operator} and \Cref{sec:study of H_T}, we compute the Fr\'{e}chet derivative of the operator $G$ at an arbitrary constant function and investigate its spectral properties. Finally, in \Cref{sec:proof_Tmain}, we provide the proof of \Cref{Th_main_B_Cap} by applying the Crandall-Rabinowitz Bifurcation Theorem.

\section{Existence result and the trivial solution}
\label{sec:baseline solution}

In this section, we first review the existence of solutions for the capillary problem over arbitrary domains. We then present the solution of the capillary problem between two parallel straight lines, which will serve as the trivial solution throughout the paper, see Figure 1(a) in \cite{MR2335074}.

Let $B=(-1,1)$. For $T>0$, we set
\begin{equation*}
  C^T_1 = B \times \mathbb{R}/T\mathbb{Z}.
\end{equation*}
An infinite strip domain in $\mathbb{R}^2$ with width $B$ can be regarded as a periodic repetition of $C^T_1$.
Given a $C^{2,\alpha}$ function $v:\mathbb{R}/ \mathbb{Z}\rightarrow \mathbb{R}$ (i.e., periodic of period $1$) with a small $C^{2,\alpha}$-norm, we define
\begin{equation*}
  C^T_{1+v}=\left\{(x,t)\in\mathbb{R}\times\mathbb{R}/T \mathbb{Z}: |x|<1+ v\left(\frac{t}{T} \right)\right\}.
\end{equation*}
Such a domain is a small perturbation of an infinite stripe with width $B$, periodic in $t$ with period $T$.

Set $u_{\nu}=\partial_{\nu} u$. Consider the following capillary problem:
\begin{equation}\label{capil_1}
  \begin{cases}
  \mathrm{div} \left(\frac{\nabla u}{\sqrt{1+|\nabla u|^2}}\right) - bu =0 &~~\mbox{in}~~ C^T_{1+v},\\
  u_{\nu}=\cos \theta \sqrt{1+ |\nabla u|^2} &~~\mbox{on}~~\partial C^T_{1+v},
  \end{cases}
\end{equation}
where $b>0$ and $0<\theta<\pi/2$.
It is well-known that for any $v$ (with a small norm, similarly hereinafter), there exists a unique solution $u$ (see \cite{Gerhardt1976global, MR2160744, MR638359, MR398278, MR816345}). When $v \equiv 0$, problem \eqref{capil_1} reduces to the classical capillary problem between two parallel straight lines, namely:
\begin{equation}\label{capi_model}
\begin{cases}
\mathrm{div} \left(\frac{\nabla u}{\sqrt{1+|\nabla u|^2}}\right) - bu=0 & \text{in } C_1^T,\\
u_{\nu}=\cos \theta \sqrt{1+ |\nabla u|^2} & \text{on } \partial C_1^T.
\end{cases}
\end{equation}
The height of the liquid in this case is constant on the boundary, i.e., $u=c>0$ on $\partial C_1^T$. Then $u_{\nu}=|\nabla u|=\cot \theta=\kappa$ on $\partial C_1^T$. Let $\Phi$ denote the solution to \eqref{capi_model}. That is, $\Phi$ is the solution of \eqref{capil} with $\Omega$ replaced by $C_1^T$. In fact, this is the start point of our bifurcation argument. By symmetry, $\Phi$ depends only on the spatial variable $x$ and is even (i.e., $\Phi(x) = \Phi(-x)$). In fact, the profile $\Phi$ can be determined by solving the following ODE:
\begin{equation}\label{capi_ODE}
\begin{cases}
\left(\frac{\Phi'}{\sqrt{1+\Phi'^2}}\right)' - b\Phi=0 & \text{in } B,\\
\Phi'(1)=\kappa, \quad
\Phi'(-1)=-\kappa.
\end{cases}
\end{equation}
We call $\Phi$ the trivial solution. For the explicit expression of $\Phi$, see Section 2.1 of \cite{MR2335074}.

\section{The construction of the bifurcation operator}
\label{sec:operator_G}

In this section, we will construct a nonlinear operator $G$ associated with the following Neumann problem
\begin{equation}\label{capil_Neumann}
\begin{cases}
\mathrm{div} \left(\frac{\nabla u}{\sqrt{1+|\nabla u|^2}}\right) - bu=0 & \text{in } C_{1+v}^T,\\
u_{\nu}=\kappa & \text{on } \partial C_{1+v}^T,
\end{cases}
\end{equation}
for any $T>0$ and $v\in C^{2,\alpha}(\mathbb{R}/ \mathbb{Z})$. The primary motivation is to transform the geometric PDE problem into an abstract functional equation of the form $G(v, T) = 0$ between appropriate Hölder spaces. In this framework, the zeros of $G$ exactly correspond to the solutions of our original problem. Establishing this operator map and analyzing its regularity properties are crucial prerequisites for rigorously applying the Crandall-Rabinowitz local bifurcation theorem in the subsequent analysis.

Our aim is to find a bifurcating curve $(v_{\eta},T_{\eta},c_{\eta})$ with $v_{\eta} \not\equiv 0$ for $\eta \neq 0$ such that problem \eqref{capil_Neumann} admits a solution with a constant boundary value. Because the physical domain $C^T_{1+v}$ varies with the unknown perturbation $v$ and the period $T$, it is highly inconvenient to apply standard functional analytic tools. Hence, we pull the problem back to a completely fixed reference cylinder $C_1^1$. In the following, we denote the coordinates in original problem as $(x,t)$, and the coordinates after transformation as $(y,s)$. We use subscripts to denote the partial derivatives, for example, $u_x:=\partial u/\partial x$. Furthermore, we adopt the standard Einstein summation convention, where summation is implied over repeated indices and the explicit sum symbol is omitted.

%Partial derivatives are represented by subscripts; specifically, $u_x$ and $u_t$ denote $\partial u/\partial x$ and $\partial u/\partial t$, while $\phi_y$ and $\phi_s$ denote $\partial \phi/\partial y$ and $\partial \phi/\partial s$, respectively.

Let us introduce the following coordinate transformation:
\begin{equation*}
  s=\frac{t}{T},\quad y=\frac{x}{1+v(s)}, \quad \phi(y,s)=u(x,t).
\end{equation*}
Thus, we have
\begin{equation*}
  \frac{\partial y}{\partial t}=-\frac{x v' T^{-1}}{(1+v)^2}=-\frac{y v' T^{-1}}{1+v}.
\end{equation*}
Then, using the chain rule, the derivatives transform as:
\begin{equation*}
  u_x = \frac{1}{1+v} \cdot \phi_y, \quad
  u_t = \frac{1}{T}\cdot \phi_s - \frac{y v' T^{-1}}{1+v} \cdot \phi_y.
\end{equation*}
Consequently, we have
\begin{equation*}
  |\nabla u|^2 = u_x^2 + u_t^2
        = \frac{1+y^2 v'^2 T^{-2}}{(1+v)^2} \cdot \phi_y^2 +\frac{1}{T^2} \cdot \phi_s^2 - \frac{2y v' T^{-2}}{1+v} \cdot \phi_y \phi_s.
\end{equation*}
Set
\begin{equation}\label{W}
  W:=\sqrt{1+|\nabla u|^2}
  =\left(1+\frac{1+y^2 v'^2 T^{-2}}{(1+v)^2} \cdot \phi_y^2 +\frac{1}{T^2} \cdot \phi_s^2 - \frac{2y v' T^{-2}}{1+v} \cdot \phi_y \phi_s\right)^{\frac{1}{2}}.
\end{equation}
A direct computation shows that
\begin{equation*}
  \left(\frac{u_x}{\sqrt{1+|\nabla u|^2}}\right)_x
    =\left(\frac{\frac{1}{1+v} \cdot \phi_y}{W}\right)_y\cdot \frac{1}{1+v}=\frac{1}{(1+v)^2}\left(\frac{\phi_y}{W}\right)_y
\end{equation*}
and
\begin{equation*}
  \begin{aligned}
    \left(\frac{u_t}{\sqrt{1+|\nabla u|^2}}\right)_t
    =&\left(\frac{\frac{1}{T}\cdot \phi_s - \frac{y v' T^{-1}}{1+v} \cdot \phi_y}{W}\right)_y\cdot \frac{-y v' T^{-1}}{1+v}
    +\left(\frac{\frac{1}{T}\cdot \phi_s - \frac{y v' T^{-1}}{1+v} \cdot \phi_y}{W}\right)_s \cdot T^{-1}\\
    =&-\frac{y v' T^{-2}}{1+v}\left(\frac{\phi_s}{W}\right)_y
    +\frac{y v'^2 T^{-2}}{(1+v)^2}\frac{\phi_y}{W}
    +\frac{y^2 v'^2 T^{-2}}{(1+v)^2}\left(\frac{\phi_y}{W}\right)_y\\
    &+\frac{1}{T^2}\left(\frac{\phi_s}{W}\right)_s
    +\frac{y v'^2 T^{-2}}{(1+v)^2}\frac{\phi_y}{W}
    -\frac{y v'' T^{-2}}{1+v}\frac{\phi_y}{W}
    +\frac{y v'^2 T^{-2}}{(1+v)^2}\frac{\phi_y}{W}
    -\frac{y v'^2 T^{-2}}{1+v}\left(\frac{\phi_y}{W}\right)_s\\
    =&-\frac{y v' T^{-2}}{1+v}\left(\frac{\phi_s}{W}\right)_y
    +\frac{3 y v'^2 T^{-2}}{(1+v)^2}\frac{\phi_y}{W}
    +\frac{y^2 v'^2 T^{-2}}{(1+v)^2}\left(\frac{\phi_y}{W}\right)_y\\
    &+\frac{1}{T^2}\left(\frac{\phi_s}{W}\right)_s
    -\frac{y v'' T^{-2}}{(1+v)^2}\frac{\phi_y}{W}
    -\frac{y v'^2 T^{-2}}{(1+v)^2}\left(\frac{\phi_y}{W}\right)_s.
  \end{aligned}
\end{equation*}

Then, the mean curvature operator transforms into:
\begin{equation}\label{Lv}
\begin{aligned}
\mathrm{div} \left(\frac{\nabla u}{\sqrt{1+|\nabla u|^2}}\right)
= & \left(\frac{u_x}{\sqrt{1+|\nabla u|^2}}\right)_x+\left(\frac{u_t}{\sqrt{1+|\nabla u|^2}}\right)_t \\
= & \frac{1+y^2 v'^2 T^{-2}}{(1+v)^2} \cdot \left( \frac{\phi_y}{W} \right)_y
   +\frac{1}{T^2} \left( \frac{\phi_s}{W} \right)_s \\
& - \frac{y v' T^{-2}}{1+v} \cdot \left( \frac{\phi_s}{W} \right)_y - \frac{y v' T^{-2}}{1+v} \cdot \left( \frac{\phi_y}{W} \right)_s \\
& - \frac{y v'' T^{-2}}{1+v} \cdot \frac{\phi_y}{W} + \frac{3 y v'^2 T^{-2}}{(1+v)^2} \cdot \frac{\phi_y}{W}\\
=:& L_{v,T} (\phi).
\end{aligned}
\end{equation}
Note that $|x|=1+v$ on $\partial C^T_{1+v}$. Thus, we have
\begin{equation*}
\begin{aligned}
  \nu= \frac{(1,-v'T^{-1})}{\sqrt{1+v'^2 T^{-2}}} \quad ~~\mbox{ on}~~\{x=1+v\}
  ~~\mbox{  and  }~~
  \nu= \frac{(-1,-v'T^{-1})}{\sqrt{1+v'^2 T^{-2}}} \quad ~~\mbox{ on}~~\{x=-(1+v)\}.
\end{aligned}
\end{equation*}
Then we obtain
\begin{equation*}
  \kappa=u_{\nu}=u_i \nu_i=
  \begin{cases}
    \frac{\sqrt{1+v'^2 T^{-2}}}{1+v}\cdot \phi_y-\frac{v' T^{-2}}{\sqrt{1+v'^2 T^{-2}}} \cdot \phi_s
  &~~\mbox{on}~~\{x=1+v\},\\
  -\frac{\sqrt{1+v'^2 T^{-2}}}{1+v}\cdot \phi_y-\frac{v' T^{-2}}{\sqrt{1+v'^2 T^{-2}}} \cdot \phi_s &~~\mbox{on}~~\{x=-(1+v)\}.
  \end{cases}
\end{equation*}
That is, in the new coordinate system,
\begin{equation*}
  \frac{y \sqrt{1+v'^2 T^{-2}}}{1+v}\cdot \phi_y-\frac{v' T^{-2}}{\sqrt{1+v'^2 T^{-2}}} \cdot \phi_s
  =\kappa ~~\quad~~\mbox{on}~~\partial C^1_{1}.
\end{equation*}
Then problem \eqref{capil_Neumann} becomes:
\begin{equation}\label{capil_4}
\begin{cases}
  L_{v,T} (\phi)- b\phi =0 &~~\mbox{in}~~ C^1_{1},\\
  \frac{y \sqrt{1+v'^2 T^{-2}}}{1+v}\cdot \phi_y-\frac{v' T^{-2}}{\sqrt{1+v'^2 T^{-2}}} \cdot \phi_s
  =\kappa &~~\mbox{on}~~\partial C^1_{1}.
\end{cases}
\end{equation}

Through this transformation, the geometric variations of the domain induced by $v$ have been entirely transferred to the coefficients of the differential operator and the boundary conditions. We now define the appropriate function spaces on this fixed reference domain. For $k \in \mathbb{N}$, let:
\begin{equation}\label{v_1_2}
\begin{aligned}
  &C_{e}^{k,\alpha}(\mathbb{R}/\mathbb{Z})=\left\{v\in C^{k,\alpha}(\mathbb{R}/\mathbb{Z}): v(-s)=v(s)\right\},\\
  &C_{e,m}^{k,\alpha}(\mathbb{R}/\mathbb{Z})=\left\{v\in C_{e}^{k,\alpha}(\mathbb{R}/\mathbb{Z}): \int_0^1 v(s) \,ds=0\right\}.
\end{aligned}
\end{equation}
In addition, we define the H\"{o}lder spaces of even functions on $C_1^1$:
\begin{equation}\label{phi_e_m}
\begin{aligned}
  &C_{e}^{k,\alpha}(C_1^1)=\left\{\phi\in C^{k,\alpha}(C_1^1): \phi(y,s)=\phi(-y,s),~\phi(y,s)=\phi(y,-s)\right\},\\
  &C_{e,m}^{k,\alpha}(C_1^1)=\left\{\phi\in C_{e}^{k,\alpha}(C_1^1): \int_0^1 \phi(1,s) \,ds=0 \right\},
\end{aligned}
\end{equation}
and the Sobolev space:
\begin{equation}\label{int_H_e_m}
\begin{aligned}
  &H_{e}^{1}(C_1^1)=\left\{\phi\in H^{1}(C_1^1):\phi(y,s)=\phi(-y,s),~\phi(y,s)=\phi(y,-s)\right\},\\
  &H_{e,m}^{1}(C_1^1)=\left\{\phi\in H_{e}^{1}(C_1^1): \int_0^1 \phi(1,s) \,ds=0 \right\},\\
  &H_{e}^{1}(B)=\left\{\phi\in H^{1}(B):\phi(y)=\phi(-y)\right\}.
\end{aligned}
\end{equation}

\medskip

Now we show that the solution $\phi$ in \eqref{capil_4} depends on $v$ and $T$ in a smooth way.
\begin{proposition} \label{Pr30}
For any $T>0$ and $v\in C_{e}^{2,\alpha}(\mathbb{R}/\mathbb{Z})$ whose norm is sufficiently small, the problem \eqref{capil_4}	has a unique positive solution $\phi=\phi_{v,T}\in C_e^{2,\alpha}(C^{1}_{1})$. Moreover, $\phi$ depends smoothly on $v$ and $T$, and $\phi_{0,T} \equiv \Phi$ for any $T>0$, where $\Phi$ is the solution of \eqref{capi_ODE}.
\end{proposition}

\begin{proof}
Given $T>0$ and $v\in C_{e}^{2,\alpha}(\mathbb{R}/\mathbb{Z})$, there exists a unique solution of \eqref{capil_Neumann} \cite[Theorem 2.1]{Gerhardt1976global}. Thus, there exists a unique solution $\phi \in C_{e}^{2,\alpha}(C^{1}_1)$ that solves \eqref{capil_4}.

Next, let us show the smooth dependence of $\phi$. For $\psi \in C_{e}^{2,\alpha}(C^{1}_1)$, we define the nonlinear mapping $\mathcal{N}: C_{e}^{2,\alpha}(\mathbb{R}/\mathbb{Z}) \times \mathbb{R}_+ \times C_{e}^{2,\alpha}(C^{1}_1) \rightarrow C_{e}^{\alpha}(C^{1}_1) \times C_{e}^{1,\alpha}(\partial C_1^1)$ by
\begin{equation} \label{defN}
\mathcal{N}(v, T, \psi) := \left(L_{v,T} (\Phi+\psi) - b(\Phi+\psi),
\frac{y \sqrt{1+v'^2 T^{-2}}}{1+v}\cdot (\Phi+\psi)_y - \frac{v' T^{-2}}{\sqrt{1+v'^2 T^{-2}}} \cdot (\Phi+\psi)_s\right),
\end{equation}
where $L_{v,T}$ is shown in \eqref{Lv}. It is clear that $\mathcal{N}$ is a smooth ($C^{\infty}$) mapping and that $\mathcal{N}(0, T, 0) = 0$ for any $T>0$, since $\Phi$ is the exact solution when the perturbation $v$ vanishes.

Now we calculate the Fréchet derivative of $\mathcal{N}$ with respect to $\psi$ evaluated at $(0, T, 0)$. By referring to \eqref{Lv}, we obtain
\begin{equation*}
  \mathcal{N}(0, T, \psi) = \left(L_{0,T} (\Phi+\psi) - b(\Phi+\psi), y (\Phi+\psi)_y \right),
\end{equation*}
where
\begin{equation*}
  L_{0,T} (\Phi+\psi)= \left( \frac{(\Phi+\psi)_y}{\left(1+(\Phi+\psi)_y^2+T^{-2}\psi_s^2\right)^{\frac{1}{2}} } \right)_y
   +\frac{1}{T^2} \left( \frac{(\Phi+\psi)_s}{\left(1+(\Phi+\psi)_y^2+T^{-2}\psi_s^2\right)^{\frac{1}{2}} } \right)_s.
\end{equation*}
Then, by a direct computation, the Fréchet derivative of $\mathcal{N}$ is a linear operator given by
\begin{equation*}
  D_{\psi}\mathcal{N}|_{(0,T,0)}(\psi) = \left(\mathfrak{L}_{0,T} (\psi) - b\psi, \psi_{\nu}\right),
\end{equation*}
where
\begin{equation*}
  \mathfrak{L}_{0,T}(\psi)
  =\left(\frac{\psi_y}{\left(1+\Phi_y^2\right)^{\frac{3}{2}}}\right)_y
  +\frac{1}{T^2}\left(\frac{\psi_s}{\left(1+\Phi_y^2\right)^{\frac{1}{2}}}\right)_s
\end{equation*}
is the linearization of the differential operator $L_{0,T}$ at the trivial solution $\Phi$. Note that $\Phi_y=\Phi'$ since $\Phi$ depends only on $y$.

Consider the following problem:
\begin{equation}\label{capil_2}
\begin{cases}
  \mathfrak{L}_{0,T} (\psi) - b\psi = f & \text{in } C^1_{1},\\
   \psi_{\nu}= g & \text{on } \partial C^1_{1}.
\end{cases}
\end{equation}
Note that the operator $\mathfrak{L}_{0,T}$ is uniformly elliptic since $\Phi_y$ is bounded. Since $b > 0$, the existence and uniqueness of solutions to this linear Neumann boundary value problem are guaranteed by standard linear elliptic theory (see, e.g., \cite[Chapter 8]{MR1814364} or \cite{MR2597943, MR3059278}). Therefore, $D_{\psi}\mathcal{N}$ is an invertible bounded linear operator at $(0,T,0)$ from $C_{e}^{2,\alpha}(C^{1}_1)$ to $C_{e}^{\alpha}(C^{1}_1) \times C_{e}^{1,\alpha}(\partial C_1^1)$. Therefore, the Implicit Function Theorem yields the existence of a unique, smooth map $(v, T) \mapsto \psi(v,T) \in C_{e}^{2,\alpha}(C^{1}_1)$, defined for $v$ in a sufficiently small neighborhood of $0$ in $C_{e}^{2,\alpha}(\mathbb{R}/\mathbb{Z})$, such that $\mathcal{N}(v, T, \psi(v,T)) = 0$.

Transforming back to the original problem, the function $\phi_{v,T} := \Phi + \psi(v,T)$ solves \eqref{capil_4}. Furthermore, its smooth dependence on $v$ and $T$ follows directly from the smoothness of the implicit function mapping.

\end{proof}

\medskip

For any $T>0$ and $v\in C_{e}^{2,\alpha}(\mathbb{R}/\mathbb{Z})$, we define the nonlinear operator $G$ as follows. Let $\mathcal{U} \subset C^{2,\alpha}_{e,m}(\mathbb{R} / \mathbb{Z})$ be a small neighborhood of $0$. Define $G: \mathcal{U}\times \mathbb{R}_+ \rightarrow C^{2,\alpha}_{e,m}(\mathbb{R} / \mathbb{Z})$ as:
\begin{equation}\label{G}
G(v,T)=\phi_{v,T}|_{\partial C^1_{1}}-\frac{1}{|\partial C^1_{1}|}\int_{\partial C^{1}_{1}} \phi_{v,T}
=\phi_{v,T}|_{\partial C^1_{1}}-\frac{1}{2}\int_{\partial C^{1}_{1}} \phi_{v,T},
\end{equation}
where $|\partial C^1_{1}|$ is the measure of $\partial C^1_{1}$ and $\phi_{v,T}$ is the solution of \eqref{capil_4}.

It is clear that $G(v,T)=0$ if and only if the solution $u$ of \eqref{capil_Neumann} is constant on $\partial C^T_{1+v}$. Hence, $u$ solves the capillary problem with constant boundary value. In addition, $G(0,T)=0$ for any $T>0$. Our goal is to identify a branch of nontrivial solutions $(v,T)$ to the equation $G(v,T)=0$ bifurcating from some point $(0, T_{*})$. To this end, we employ a local bifurcation argument. This approach naturally leads to the analysis of the linearization of $G$ at $(0,T)$, which will be carried out in the next section.

\section{The linearization of the operator $G$}
\label{sec:linearized_operator}

In this section, we will compute the Fr\'{e}chet derivative of the operator $G$ with respect to $v$ at $(0,T)$. We denote such derivative by $H_T$, i.e.,
\begin{equation*}
  H_T:=D_v G(0,T).
\end{equation*}

In fact, in the following, given $w\neq 0$, we will compute the directional derivative along $w$, i.e., $H_T(w)$. Set $v(s)=\eta w(s)$ where $\eta \in \mathbb{R}$.
By noting that \eqref{W} and \eqref{Lv}, we have
\begin{equation*}
  W=\left(1+\frac{1+y^2 \eta^2 w'^2 T^{-2}}{(1+\eta w)^2} \cdot \phi_y^2 + \frac{1}{T^2} \phi_s^2 - \frac{2y \eta w' T^{-2}}{1+\eta w} \cdot \phi_y \phi_s\right)^{\frac{1}{2}}.
\end{equation*}
and
\begin{align*}
L_{\eta, w,T} (\phi):=
& \frac{1+y^2 \eta^2 w'^2 T^{-2}}{(1+\eta w)^2} \cdot \left( \frac{\phi_y}{W} \right)_y
   + \frac{1}{T^2} \cdot \left( \frac{\phi_s}{W} \right)_s \\
& - \frac{y \eta w' T^{-2}}{1+\eta w} \cdot \left( \frac{\phi_s}{W} \right)_y - \frac{y \eta w' T^{-2}}{1+\eta w} \cdot \left( \frac{\phi_y}{W} \right)_s \\
& - \frac{y \eta w'' T^{-2}}{1+\eta w} \cdot \frac{\phi_y}{W} + \frac{3 y \eta^2 w'^2 T^{-2}}{(1+\eta w)^2} \cdot \frac{\phi_y}{W}.
\end{align*}
Then, we have that $\phi$ satisfies the following:
\begin{equation}\label{capil_3}
\begin{cases}
  L_{\eta, w,T} (\phi)- b\phi =0 &~~\mbox{in}~~ C^1_{1},\\
  \frac{y \sqrt{1+\eta^2 w'^2 T^{-2}}}{1+\eta w}\cdot \phi_y-\frac{\eta w' T^{-2}}{\sqrt{1+\eta^2 w'^2 T^{-2}}} \cdot \phi_s=\kappa&~~\mbox{on}~~\partial C^1_{1}.
\end{cases}
\end{equation}

\medskip

Note that if $\eta=0$, \eqref{capil_3} has the same solution as \eqref{capi_model}, i.e. $\phi |_{\eta=0}=\Phi$, and $\Phi_s=0$. Thus, we obtain
\begin{equation*}
W|_{\eta=0}=\left(1+\Phi_y^2\right)^{\frac{1}{2}}.
\end{equation*}

Let
\begin{equation*}
  \dot{\phi}=\frac{d \phi}{d \eta}\bigg|_{\eta=0}, \quad \dot{W}=\frac{d W}{d \eta}\bigg|_{\eta=0}.
\end{equation*}
Then we obtain
\begin{equation*}
  \dot{W}=\left(1+\Phi_y^2\right)^{-\frac{1}{2}}
  \left(\Phi_y \dot \phi_y - \Phi_y^2 w\right).
\end{equation*}

Note that $\Phi$ depends only on $y$. Then, by taking the derivative in \eqref{capil_3} with respect to $\eta$ and letting $\eta=0$, we obtain that $\dot{\phi}$ satisfies
\begin{equation}\label{phi_dot}
\begin{cases}
  \mathcal{L}_{w,T} (\dot \phi)- b\dot \phi =0 &~~\mbox{in}~~ C^1_{1},\\
  \dot \phi_y-w \Phi_y=0 &~~\mbox{on}~~\partial C^1_{1},
\end{cases}
\end{equation}
where
\begin{align*}
  \mathcal{L}_{w,T} (\dot \phi)=&\left(\frac{\dot \phi_y}{(1+\Phi_y^2)^{\frac{3}{2}}}\right)_y
  +\left( \frac{\Phi_y^3 w}{(1+\Phi_y^2)^{\frac{3}{2}}}\right)_y
  -2w\left(\frac{\Phi_y}{(1+\Phi_y^2)^{\frac{1}{2}}}\right)_y\\
  &+T^{-2} \left( \frac{\dot \phi_s}{(1+\Phi_y^2)^{\frac{1}{2}}}\right)_s
  -yw'' T^{-2} \cdot \frac{\Phi_y}{(1+\Phi_y^2)^{\frac{1}{2}}}.
\end{align*}
Let $\hat \psi=\dot \phi-yw\Phi_y$. Then we have
\begin{equation}\label{psi_y_s}
  \hat \psi_y=\dot \phi_y -yw\Phi_{yy}-w\Phi_y,\quad
  \hat \psi_s=\dot \phi_s-yw'\Phi_y.
\end{equation}
By inserting \eqref{psi_y_s} into \eqref{phi_dot}, we obtain
\begin{equation}\label{psi_hat_1}
\begin{cases}
  \hat {\mathcal{L}}_{w,T} (\hat \psi)- b(\hat \psi+yw\Phi_y)=0 &~~\mbox{in}~~ C^1_{1},\\
  \hat \psi_y+yw\Phi_{yy}=0 &~~\mbox{on}~~\partial C^1_{1},
\end{cases}
\end{equation}
where
\begin{align*}
   \hat {\mathcal{L}}_{w,T} (\hat \psi)=
   \left( \frac{\hat \psi_y}{(1+\Phi_y^2)^{\frac{3}{2}}}\right)_y
  +T^{-2} \left( \frac{\hat \psi_s}{(1+\Phi_y^2)^{\frac{1}{2}}}\right)_s
   +\left( \frac{yw \Phi_{yy}}{(1+\Phi_y^2)^{\frac{3}{2}}}\right)_y
   -w\left(\frac{\Phi_y}{(1+\Phi_y^2)^{\frac{1}{2}}}\right)_y.
\end{align*}

Recall that when $\eta=0$, $\Phi$ is the solution of \eqref{capi_model}, i.e., $\Phi$ satisfies
\begin{equation}\label{b_phi}
  \left(\frac{\Phi_{y}}{(1+\Phi_y^2)^{\frac{1}{2}}}\right)_y-b\Phi
  =\frac{\Phi_{yy}}{(1+\Phi_y^2)^{\frac{3}{2}}}-b\Phi=0.
\end{equation}
By taking derivatives with respect to $y$, we obtain
\begin{equation}\label{eq_y}
  \left( \frac{\Phi_{yy}}{(1+\Phi_y^2)^{\frac{3}{2}}}\right)_y-b\Phi_y=0.
\end{equation}
%That is,
%\begin{equation}\label{eq_y2}
%  ywb\phi_y=yw\left(\frac{1}{1+\phi_y^2}\cdot \frac{\phi_{yy}}{\sqrt{1+\phi_y^2}}\right)_y
%\end{equation}
By combing \eqref{b_phi} and \eqref{eq_y}, \eqref{psi_hat_1} becomes
\begin{equation}\label{psi_hat}
\begin{cases}
  \mathrm{div} (A\cdot \nabla \hat \psi)- b \hat \psi =0 &~~\mbox{in}~~ C^1_{1},\\
  \hat \psi_y+b\Phi \left(1+\Phi_y^2\right)^{\frac{3}{2}} yw=0 &~~\mbox{on}~~\partial C^1_{1},
\end{cases}
\end{equation}
where
\begin{equation} \label{matrix}
	A=
\left(
\begin{matrix}
\left(1+\Phi_y^2\right)^{-\frac{3}{2}} & 0\\
0 & T^{-2} \left(1+\Phi_y^2\right)^{-\frac{1}{2}}
\end{matrix}
\right).
\end{equation}
Note that $A$ depends only on $y$ and $T$. Since $\Phi=c$ and $\Phi_y^2=\kappa^2$ on $\partial C_1^1$, \eqref{psi_hat} reduces to
\begin{equation}\label{psi_hat_2}
\begin{cases}
  \mathrm{div} (A\cdot \nabla \hat \psi)- b \hat \psi =0 &~~\mbox{in}~~ C^1_{1},\\
  \hat \psi_{\nu}+\hat c w=0 &~~\mbox{on}~~\partial C^1_{1},
\end{cases}
\end{equation}
where $\hat c=bc \left(1+\kappa^2\right)^{3/2} >0$ and $\hat \psi_{\nu}$ represents as usual the normal derivative about $\partial C^1_{1}$.

\medskip

In the following, we verify that $\hat{\psi} \in C_{e,m}^{2,\alpha}(C_1^1)$. Since $w\in C^{2,\alpha}$ and $A$ is smooth, by the standard regularity theory, $\hat{\psi} \in C^{2,\alpha}(C_1^1)$ (in fact, $\hat{\psi} \in C^{3,\alpha}(C_1^1)$). Note that $\hat{\psi}(-y,s)$ and $\hat{\psi}(y,-s)$ are also solutions of \eqref{psi_hat_2}. By the uniqueness, $\hat{\psi}(y,s)=\hat{\psi}(-y,s)=\hat{\psi}(y,-s)$. That is, $\hat{\psi}$ is even both in $y$ and $s$.

By noting $w\in C_{e,m}^{2,\alpha}(\mathbb{R}/\mathbb{Z})$, we have
\begin{equation}\label{int_w}
  \int_{0}^{1} w dt =0.
\end{equation}
We define $\Psi \in C_e^{2,\alpha}(C_1^1)$ as the unique solution of the following problem
\begin{equation}\label{Psi}
\begin{cases}
  \mathrm{div} \left(A \cdot \nabla \Psi \right)-b\Psi=0 &~~\mbox{in}~~ C^1_{1},\\
  \Psi=1 &~~\mbox{on}~~\partial C^1_{1}.
\end{cases}
\end{equation}
The existence and uniqueness of the solution $\Psi$ for the Dirichlet problem \eqref{Psi} is a consequence of the classical Schauder theory for linear uniformly elliptic equations (see, e.g., \cite[Theorem 6.13]{MR1814364} or \cite[Chapter 4]{MR2777537}). Multiplying the equation in \eqref{Psi} by $\hat \psi$ and the equation in \eqref{psi_hat_2} by $\Psi$, followed by integration by parts in both expressions, we obtain
\begin{equation*}
  \int_{\partial C_1^1} A^{11} \Psi_{\nu} \hat \psi =\int_{\partial C_1^1} A^{11} \hat \psi_v \Psi.
\end{equation*}
Note that $\Psi$, $\Psi_{\nu}$ and $A^{11}$ are positive constants on $\partial C_1^1$. Then, by the boundary condition of \eqref{psi_hat_2} , there exists a positive constant $C$ such that
\begin{equation*}
  \int_{\partial C_1^1} \hat \psi =-C \hat c\int_{\partial C_1^1} w.
\end{equation*}
By considering \eqref{int_w}, we obtain
\begin{equation}\label{int_hat-psi}
  \int_{\partial C_1^1} \hat \psi =0.
\end{equation}
Therefore, $\hat{\psi} \in C_{e,m}^{2,\alpha}(C_1^1)$.

Finally, by recalling the expression of G in \eqref{G}, $\hat \psi =\dot \phi -y w \Phi_y$ and
\begin{equation*}
  y \Phi_y=\Phi_{\nu}=\kappa~~\mbox{ on }~~\partial C^1_{1},
\end{equation*}
we obtain
\begin{equation}\label{H_T}
  \begin{aligned}
    H_T(w)=&D_{v} G |_{v=0} (w)=\frac{d}{d \eta} G (\eta w, T)  \bigg |_{\eta=0}\\
    =&\dot \phi |_{\partial C^1_{1}}-\frac{1}{2}\int_{\partial C^{1}_{1}} \dot \phi\\
    =&\hat \psi |_{\partial C^1_{1}}+\kappa w
    -\frac{1}{2}\int_{\partial C^{1}_{1}} \left(\hat \psi+\kappa w\right)\\
    =&\hat \psi |_{\partial C^1_{1}}+\kappa w.
  \end{aligned}
\end{equation}
%Since $\hat{\psi}$ satisfies \eqref{psi_hat_2} and \eqref{int_hat-psi}, it can be verified that $\hat{\psi} \in C_{e,m}^{2,\alpha}(C_1^1)$.

\medskip

Next, we present some properties of $H_T$.

\begin{lemma} \label{H_T_Fredholm}
For any $T>0$, the operator
\begin{equation*}
  H_{T}:C_{e,m}^{2,\alpha}(\mathbb{R}/\mathbb{Z})\rightarrow C_{e,m}^{2,\alpha}(\mathbb{R}/\mathbb{Z})
\end{equation*}
is a linear self-adjoint operator and a Fredholm operator of index zero.
 \end{lemma}

\begin{proof}
Given $v_i \in C_{e,m}^{2,\alpha}(\mathbb{R}/\mathbb{Z})$, $i=1,2$, there exist unique $\hat \psi (v_i)$ satisfies \eqref{psi_hat_2} (with $w$ replaced by $v_i$). For $\hat \psi (v_1)$, we multiply the equation in \eqref{psi_hat_2} by $\hat \psi (v_2)$ and integrate by parts to obtain
\begin{equation}\label{v_i_e}
  \int_{C_1^1} \left(A^{ij} \partial_j \hat \psi (v_1) \partial_i \hat \psi (v_2)
  +b\hat \psi(v_1)\hat \psi(v_2)\right)
  =\int_{\partial C_1^1} A^{11}\partial_{\nu} \hat \psi(v_1) \hat \psi (v_2).
\end{equation}
In addition, from the boundary condition of \eqref{psi_hat_2}, we get
\begin{equation}\label{v_i_b}
  v_i=-\frac{\partial_{\nu} \hat \psi(v_i)}{\hat c}
  =-\frac{A^{11} \partial_{\nu} \hat \psi(v_i)}{b c}.
\end{equation}

Then, by combining \eqref{H_T}, \eqref{v_i_e} and \eqref{v_i_b}, we have
\begin{equation*}
  \begin{aligned}
  &\int_0^1  H_{T}(v_1) v_2-\int_0^1 H_{T}(v_2) v_1\\
=&\int_{\partial C_1^1} \left(\hat \psi(v_1) v_2+\kappa v_1v_2\right)
  -\int_{\partial C_1^1} \left(\hat \psi(v_2) v_1+\kappa v_2 v_1\right)\\
=&-\int_{\partial C_1^1} \left(\hat \psi(v_1) \frac{A^{11} \hat \psi_{\nu}(v_2)}{b c}
        -\hat \psi(v_2) \frac{A^{11} \hat \psi_{\nu}(v_1)}{b c}\right)\\
=&-\frac{1}{bc}\int_{C_1^1} \left(A^{ij} \hat \psi_i (v_1) \cdot \hat \psi_j(v_2)+b\hat \psi(v_1) \hat \psi(v_2)
  -A^{ij} \hat \psi_i(v_2) \cdot \hat \psi_j(v_1)-b\hat \psi(v_1)\hat \psi (v_2)\right)\\
=&0.
    \end{aligned}
\end{equation*}
Therefore, the operator $H_{T}$ is self-adjoint.

By the regularity for Neumann boundary condition (see \cite[Theorem 6.30]{MR1814364}), $\hat{\psi}\in C^{3,\alpha}$. As $C^{3,\alpha}$ is compactly embedded into $C^{2,\alpha}$, $H_T$ can be written as the sum of a compact operator and a constant multiple of the identity operator. Hence, $H_T$ is a Fredholm operator of index zero (see \cite[Theorem 6.6]{MR2759829}).
\end{proof}

\section{Study of the linearized operator $H_T$}
\label{sec:study of H_T}

A bifurcation of the branch $(0,T)$ of solutions of $G(v, T)=0$ might appear only at points $(0,T_{*})$ such that $H_{T_{*}}$ becomes degenerate. In this section, we analyze the linearized operator $H_T$ to identify a critical value $T_{*}$ such that $H_{T_{*}}(w)=0$ for some $w\neq 0$.

Recall that for any given $w$, there exists a unique $\hat \psi$ satisfying \eqref{psi_hat_2}, such that the linearized operator is $H_T(w)=\hat \psi |_{\partial C^1_{1}}+\kappa w$. Seeking a non-trivial function $w \neq 0$ that satisfies $H_T(w)=0$ is equivalent to seeking a nonzero solution in $C_{e,m}^{2,\alpha}(C_1^1)$ (see the definition of \eqref{phi_e_m}) of the following corresponding problem:
\begin{equation}\label{capil_egen_1}
\begin{cases}
  \mathrm{div} (A\cdot \nabla \psi)- b \psi =0 &~~\mbox{in}~~ C_1^1,\\
  \psi_{\nu}-\bar c \psi=0 &~~\mbox{on}~~\partial C_1^1,
\end{cases}
\end{equation}
where
\begin{equation}\label{e5.13}
\bar c=\frac{\hat c}{\kappa}=\frac{bc \left(1+\kappa^2\right)^{\frac{3}{2}}}{\kappa}>0.
\end{equation}

To clarify this equivalence, suppose that there exists $w\neq 0$ such that $H_T(w)=0$, i.e., $\hat \psi+\kappa w=0$ on $\partial C_1^1$, where $\hat{\psi}$ is the solution of \eqref{psi_hat_2}. Then, by combining with the boundary condition in \eqref{psi_hat_2}, we obtain the boundary condition in \eqref{capil_egen_1}. Hence, $\hat{\psi}\in C_{e,m}^{2,\alpha}(C_1^1)$ is a nonzero solution of \eqref{capil_egen_1}.
Conversely, suppose that $\psi \in C_{e,m}^{2,\alpha}(C_1^1)$ is a nonzero solution of \eqref{capil_egen_1}. Set $w=-\psi/\kappa$ on $\partial C_1^1$. It follows that $w$ has mean zero and $\psi$ satisfies \eqref{psi_hat_2}, and hence $H_T(w)=\psi |_{\partial C_1^1} +\kappa w=\psi |_{\partial C_1^1} -\psi |_{\partial C_1^1}=0$ as required.

Based on the above discussion, let us consider the following corresponding eigenvalue problem:
\begin{equation}\label{capil_egen}
\begin{cases}
  \mathrm{div} (A\cdot \nabla \psi)- b \psi +\lambda \psi =0 &~~\mbox{in}~~ C_1^1,\\
  \psi_{\nu}-\bar c \psi=0 &~~\mbox{on}~~\partial C_1^1.
\end{cases}
\end{equation}
Finding a nonzero solution to \eqref{capil_egen_1} is equivalent to showing that $0$ is an eigenvalue of \eqref{capil_egen}. In what follows, we demonstrate that $0$ is indeed the principal eigenvalue by choosing a proper $T$.
If we multiply the first equation of \eqref{capil_egen} by $\psi$ and we integrate by parts using the boundary condition, we get
\[
\int_{C_1^1}
A^{ij}\psi_i\psi_j
+
b\int_{C_1^1}\psi^2
-
\frac{bc}{\kappa}
\int_{\partial C_1^1}\psi^2
=
\lambda
\int_{C_1^1}\psi^2\,.
\]
Define then the quadratic form $Q^{T}: H_{e,m}^1(C^1_1) \rightarrow\mathbb{R}$,
\begin{equation}\label{QT_1}
Q^{T}(\psi):=\int_{C^1_1}\left(A^{ij} \psi_i \psi_j+b\psi^2 \right)
- \frac{bc}{\kappa} \int_{\partial C^1_1} \psi^2.
\end{equation}
The principal eigenvalue of \eqref{capil_egen} is given by
\begin{equation}\label{lambda_1}
  \lambda_1^T= \inf \left\{Q^{T}(\psi): \psi\in H^1_{e,m}(C^1_1),~\|\psi\|_{L^2(C_1^1)}=1 \right\}.
\end{equation}

\medskip

Next, it will be useful to define the quadratic form given by the restriction of $Q^T$ to functions independent on the variable $s$, i.e., the quadratic form $Q: H_{e}^1(B) \rightarrow\mathbb{R}$,
\begin{equation}\label{Q_B_1}
Q(\psi):=\int_{B}\left(A^{11} {\psi'} ^2+b\psi^2 \right)
- \frac{2bc}{\kappa} {\psi}^2(1),
\end{equation}
where $A^{11}=\left(1+{\Phi'}^2\right)^{-3/2}$ (see \eqref{matrix}).
In addition, we define
\begin{equation}\label{tau_1}
\tau_1=\inf \left\{Q(\psi): \psi\in H^1_{e}(B),~\|\psi\|_{L^2(B)}=1\right\}.
\end{equation}
It is standard that $\tau_1$ is the principal eigenvalue of the following problem:
\begin{equation}\label{B_egen}
\begin{cases}
  (A^{11} \psi')'- b \psi +\tau \psi =0 &~~\mbox{in}~~ B,\\
  \psi_{\nu}-\bar c \psi=0 &~~\mbox{on}~~\partial B.
\end{cases}
\end{equation}

\medskip

Now, we characterize the eigenvalue $\tau_1$.

\begin{lemma} \label{tau_1_neg}
There holds: $\tau_1 < 0$.
 \end{lemma}
\begin{proof}
Recall that the trivial profile $\Phi$ satisfies the ODE problem \eqref{capi_ODE}. Differentiating this equation with respect to $x$ yields
\begin{equation*}
  \left(\frac{\Phi''}{(1+{\Phi'}^2)^{\frac{3}{2}}}\right)' - b\Phi' = 0.
\end{equation*}
Multiplying this equation by $\Phi'$ and integrating over $B = (-1, 1)$, we apply integration by parts to obtain
\begin{equation}\label{Phi_y_1}
\int_{B} \frac{{\Phi''}^2}{\left(1+{\Phi'}^2\right)^{\frac{3}{2}}}
+\int_{B} b {\Phi'}^2
-\left(\frac{\Phi'(1) \Phi''(1)}{\left(1+{\Phi'}^2(1)\right)^{\frac{3}{2}}}
-\frac{\Phi'(-1) \Phi''(-1)}{\left(1+{\Phi'}^2(-1)\right)^{\frac{3}{2}}}\right)
=0.
\end{equation}
Furthermore, multiply $\Phi'$ to the equation \eqref{capi_ODE}, we arrive at:
\begin{equation}\label{model_Phi_y}
  \Phi' \Phi'' (1+{\Phi'}^2)^{-\frac{3}{2}} = b\Phi \Phi'.
\end{equation}
%Evaluating this relation at the boundary points $x = \pm 1$ and utilizing the symmetry of the profile, we observe that the second derivative at the boundary is given by
%\begin{equation}\label{model_Phi_y_2}
%  \Phi''(\pm 1) = b\Phi(\pm 1) (1+{\Phi'}^2(\pm 1))^{\frac{3}{2}}.
%\end{equation}
By substituting the boundary evaluations of \eqref{model_Phi_y} back into the boundary terms of \eqref{Phi_y_1}, the energy identity simplifies to
\begin{equation}\label{Phi_y_2}
  \int_{B} \frac{{\Phi''}^2}{(1+{\Phi'}^2)^{\frac{3}{2}}}  + \int_{B} b {\Phi'}^2  = b \left( \Phi(1)\Phi'(1) - \Phi(-1)\Phi'(-1) \right).
\end{equation}

Now we evaluate the one-dimensional quadratic form $Q$ using the test function
\begin{equation*}
  \Psi(y)=
  \begin{cases}
  \Phi'(y) & ~~ y\in (0,1],\\
  -\Phi'(y) & ~~ y\in [-1,0].
  \end{cases}
\end{equation*}
By noting $\Phi'(0)=0$, $\Psi$ is Lipschitz continuous. Hence, we can use $\Psi$ as a test function. Recalling the definition of the coefficient $A^{11} = (1+{\Phi'}^2)^{-3/2}$, we compute:
\begin{equation}\label{QT_Phiy}
\begin{aligned}
  Q(\Psi) &= \int_{B} \big( A^{11} {\Psi'}^2 + b{\Psi}^2 \big)  - \frac{bc}{\kappa} \left({\Psi}^2(1) + {\Psi}^2(-1)\right) \\
  &= \int_{B} \big( A^{11} {\Phi''}^2 + b{\Phi'}^2 \big)  - \frac{bc}{\kappa} \left({\Phi'}^2(1) + {\Phi'}^2(-1)\right) \\
  &= \int_{B} \left( \frac{{\Phi''}^2}{(1+{\Phi'}^2)^{\frac{3}{2}}} + b{\Phi'}^2 \right)  -2bc \kappa \\
  &=2bc \kappa-2bc \kappa=0.
\end{aligned}
\end{equation}

%Since the test function $\Psi$ yields a zero Rayleigh quotient, but is not smooth at $0$, it cannot correspond to the principal eigenvalue $\tau_1$. By the strict variational characterization of the lowest eigenvalue, it follows that $\tau_1 < 0$.

Since the test function $\Psi$ yields a Rayleigh quotient equal to zero, we obtain $\tau_1 \leq 0$. If it was $\tau_1 =0$ then $\Psi$ would have to be a first eigenfunction. However, first eigenfunctions are smooth, whereas $\Psi$ is not smooth at 0. Therefore $\tau_1 \neq 0$ and then $\tau_1 < 0$.

\end{proof}

We now turn to the analysis of $\lambda_1^T$.
\begin{proposition} \label{Pr40}
There exists $T_{*}>0$ such that $\lambda_1^T=0$ at $T_*$. Moreover, $\emph{Ker}(H_{T_{*}})=\mathbb{R}\cos(2\pi s)$. In particular, $\emph{dim Ker}(H_{T_{*}})=1.$
\end{proposition}

\begin{proof}
In this proof, subscripts are used as indices and do not indicate derivatives. By the classical calculus of variation, $Q^T$ attains its minimum $\lambda_1^T$ at some $\psi^T$. Since $\psi^T$ is even in $s$, $\psi^T$ admits the following Fourier cosine series expansion:
\begin{equation*}
  \psi^T(y,s)=\sum_{k=0}^{\infty} a_k\psi_k (y) \cos \left(2\pi k s\right),
\end{equation*}
where $a_k$ are constants and
\begin{equation*}
\left \|\psi_k (y) \cos \left(2\pi k s\right) \right\|_{L^2(C_1^1)}=1,~\forall ~k\geq 0.
\end{equation*}
Equivalently,
\begin{equation*}
\int_{-1}^{1}\psi^2_0=1,\quad
\int_{-1}^{1}\psi^2_k=2,~\forall ~k\geq 1.
\end{equation*}
Since we require that $\int_{\partial C^1_1}\psi^T=0$, the $\psi_0$ must satisfy
\begin{equation}\label{e5.2}
\psi_0(1)=\psi_0(-1)=0.
\end{equation}
In addition,
\begin{equation}\label{e5.3}
\left \|\psi^T \right\|_{L^2(C_1^1)}^2=\sum_{k=0}^{\infty}a_k^2=1.
\end{equation}

Substituting the above expansion into \eqref{QT_1} and utilizing the orthogonality of the trigonometric system over the period $T$ (which causes all cross terms to vanish), we obtain
\begin{equation}\label{e5.1}
Q^{T}(\psi^T)=\sum_{k=0}^{\infty} Q^T\left(a_k\psi_k  \cos \left(2\pi k s\right)\right)
=\sum_{k=0}^{\infty} a_k^2Q^T_k(\psi_k),
\end{equation}
where we have defined
\[
Q^T_k(\psi) = Q^T\left(\psi\,  \cos \left(2\pi k t\right)\right)
\]
for $\psi \in H^1_e(B)$. Using \cref{e5.2}, we have
\begin{equation}\label{e5.4}
Q^{T}_0(\psi_0)= \int_{C_1^1} A^{11} \psi_0'^2+b\int_{C_1^1}\psi_0^2
= \int_{-1}^{1} A^{11} \psi_0'^2+b.
\end{equation}
In addition, for $k\geq 1$,
\begin{equation}\label{e5.6}
  \begin{aligned}
    Q^{T}_k(\psi_k)=& \int_{C^{1}_{1}} A^{11}  \psi_k'^2 \cos^2 \left(2\pi k s\right)
    +\int_{C^{1}_{1}} A^{22} \left(2\pi k \right)^2 \psi_k^2 \sin^2 \left(2\pi k s\right)\\
    &+b \int_{C^{1}_{1}} \psi_k^2 \cos^2 \left(2\pi k s\right)
    -\frac{bc}{\kappa} \int_{\partial C^{1}_{1}} \psi_k^2 \cos^2 \left(2\pi k s\right)\\
    =& \frac{1}{2} \int_{-1}^{1} A^{11}\psi_k'^2
    +2 \pi^2k^2 \int_{-1}^{1} A^{22}\psi_k^2+\frac{b}{2} \int_{-1}^{1} \psi_k^2 -\frac{bc}{\kappa} \psi_k^2(1).
  \end{aligned}
\end{equation}
Recall that $A^{11}=\left(1+{\Phi'}^2\right)^{-3/2}$ and $A^{22}=T^{-2} \left(1+{\Phi'}^2\right)^{-1/2}$ (see \eqref{matrix}).

From the expression of $Q^{T}_k$, we know that
\begin{equation*}
Q^{T}_k(\psi)- Q^{T}_1(\psi)=2 \pi^2(k^2-1) \int_{-1}^{1} A^{22}\psi^2\geq 0,~~~~\forall ~\psi\in H_e^1(B),~\forall ~k\geq 2.
\end{equation*}
Since $\psi^T$ realized the minimum of $Q^T$, from this last inequality we have that $a_k\equiv 0$ for any $k\geq 2$. That is, $\psi^T$ can be written as
\begin{equation*}
  \psi^T(y,s)=a_0\psi_0(y)+a_1\psi_1 (y) \cos \left(2\pi s\right).
\end{equation*}
Furthermore, recall the variational characterization of the principal eigenvalue $\tau_1$ in \eqref{tau_1}. Then
\begin{equation*}
\int_{-1}^{1} \left( A^{11}\psi_1'^2 + b\psi_1^2 \right) - \frac{2bc}{\kappa}\psi_1^2(1)
\geq \tau_1 \int_{-1}^{1} \psi_1^2.
\end{equation*}
%Multiplying this inequality by $\frac{T}{2}$ exactly matches the remaining terms in our energy functional. By further assuming the standard normalization condition
%\begin{equation*}
%  \frac{T}{2}\int_{-1}^1 \psi_1^2 = \int_{C^T_1} \psi_1^2 \cos^2\left(\frac{2\pi}{T}t\right) = 1,
%\end{equation*}
%we can simplify the lower bound. Thus, retaining only the first mode ($k=1$), we deduce
Hence,
\begin{equation}\label{e5.7}
  \begin{aligned}
Q^T_1(\psi_1)= & \frac{1}{2} \int_{-1}^{1}A^{11}\psi_1'^2
    +2 \pi^2 \int_{-1}^{1} A^{22}  \psi_1^2
    +\frac{b }{2} \int_{-1}^{1} \psi_1^2-\frac{bc }{\kappa} \psi_1^2(1) \\
    \geq &\frac{\tau_1}{2} \int_{-1}^{1} \psi_1^2
    +2 \pi^2 \int_{-1}^{1} A^{22}\psi_1^2 \\
    =& \tau_1 +2 \pi^2 \int_{-1}^{1} A^{22}  \psi_1^2\\
    =& \tau_1 +\frac{2 \pi^2}{T^2} \int_{-1}^{1} \left(1+\Phi'^2\right)^{-\frac{1}{2}}  \psi_1^2
  \end{aligned}
\end{equation}
and the equality holds if we take $\psi_1\equiv \psi^*$ where $\psi^*$ is the eigenfunction of $Q$ corresponding to $\tau_1$. Since $\tau_1<0$ and $\left(1+\Phi'^2\right)^{-1/2}\leq 1$, for $T$ large enough, we have
\begin{equation}\label{e5.8}
  \begin{aligned}
Q^T_1(\psi^*)=& \tau_1 +\frac{2 \pi^2}{T^2} \int_{-1}^{1} \left(1+\Phi'^2\right)^{-\frac{1}{2}}  \psi^{*2}
\leq \tau_1+\frac{2 \pi^2}{T^2} \int_{-1}^{1} \psi^{*2}
=\tau_1+\frac{4 \pi^2}{T^2}<0.
  \end{aligned}
\end{equation}
Hence, if we take $\psi=\psi^* (y) \cos \left(2\pi s\right)$, then
\begin{equation*}
Q^T(\psi)=Q^T_1(\psi^*)<0.
\end{equation*}
Thus,
\begin{equation}\label{e5.9}
\lambda_1^T=Q^T(\psi^T)\leq Q^T(\psi)<0.
\end{equation}

On the other hand, if $T$ is small enough, by noting $\left(1+\Phi'^2\right)^{-1/2}\geq C$ for some constant $C$, we have
\begin{equation}\label{e5.10}
  \begin{aligned}
Q^T_1(\psi_1)\geq & \tau_1 +\frac{2 \pi^2}{T^2} \int_{-1}^{1} \left(1+\Phi'^2\right)^{-\frac{1}{2}}  \psi_1^2
\geq \tau_1+\frac{2C \pi^2}{T^2} \int_{-1}^{1}\psi_1^2
=\tau_1+\frac{4C \pi^2}{T^2}\geq 1.
  \end{aligned}
\end{equation}
Moreover, note that $Q^T_0(\psi_0)\geq b>0$. Hence, for $T$ small enough,
\begin{equation}\label{e5.11}
  \begin{aligned}
    \lambda_1^T=Q^{T}(\psi^T)=a_0^2Q^T_0(\psi_0)+a_1^2Q^T_1(\psi_1)
    \geq a_0^2b+a_1^2 \geq  \min(1,b)>0.
  \end{aligned}
\end{equation}

Since $\lambda_1^T$ depends on $T$ continuously, by combing \cref{e5.9} and \cref{e5.11}, there exists $T_*>0$ such that
\begin{equation}\label{e5.12}
\lambda_1^T=0~~\mbox{ at }~~T_*.
\end{equation}
Moreover, at $T_*$, we have $a_0=0$ and $a_1=1$ since $Q_0(\psi_0)\geq b>0$. That is, $\psi^{T_*}$ must be in the form
\begin{equation}\label{psi_T_8}
\psi^{T_*}(y,s)=\psi_1 (y) \cos \left(2\pi s\right).
\end{equation}

Note that $\psi_1>0$ or $\psi_1<0$ in $B$. Otherwise, if $\psi_1(1)\geq 0$, let
\begin{equation*}
\psi(y,s)=\max(\psi_1(y),0)\cos(2\pi s).
\end{equation*}
Then
\begin{equation*}
Q^{T_*}(\psi)=Q_1^{T_*}(\psi)<Q_1^{T_*}(\psi_1)=Q^{T_*}(\psi^{T_*}),
\end{equation*}
which is a contradiction. If $\psi_1(1)<0$, by considering
\begin{equation*}
\psi(y,s)=\min(\psi_1(y),0)\cos(2\pi s),
\end{equation*}
we have a contradiction in a similar way. Hence, $\psi$ has a constant sign in $B$. Without loss of generality, we assume that $\psi_1>0$ in $B$.
\end{proof}

\section{Bifurcation argument and proof of the main result}
\label{sec:proof_Tmain}

In this section we conclude the proof of our main result by means of the classical Crandall-Rabinowitz Theorem \cite{MR2859263}.

\begin{theorem} \label{CR1} \textbf{\mbox{(Crandall-Rabinowitz Bifurcation Theorem)}}
	Let $X$ and $Y$ be Banach spaces, and let $U\subset X$ and $I\subset\mathbb{R}$ be open sets, where we assume $0\in U$. Denote the elements of $U$ by $v$ and the elements of $I$ by $T$. Let $G: U \times I \rightarrow Y$ be a $C^{2}$ operator such that
	\begin{itemize}
		\item[(i)] $G(0,T)=0$ for all $T \in I,$
		\item[(ii)] $\ker D_{v}G(0,T_*)=\mathbb{R}\,w$ for some $T_* \in I$ and some $w\in X\setminus\{0\};$
		\item[(iii)] $\mathrm{codim}\, \mathrm{Im} D_{v}G(0, T_*)=1;$
		\item[(iv)] $D_{T}D_{v}G(0, T_*)(w)\notin \emph{Im~} D_{v}G(0, T_*).$
	\end{itemize}
	Then there exists a nontrivial $C^2$ curve
	\begin{equation}\label{eq.B_Cap}
		(-\varepsilon, \varepsilon) \ni \eta \mapsto (v_{\eta}, T_{\eta}) \in X \times I,
	\end{equation}
for some $\varepsilon>0$, such that:
\begin{enumerate}
	\item $v_{0}=0$, $v'_{0}=w$, $T_{0}=T_*$, $T'_{0}=0$,.
	\item $G\left(v_{\eta}, T_{\eta}\right)=0$ for all $ \eta \in(-\varepsilon,\varepsilon)$.
\end{enumerate}
	Moreover, there exists a neighborhood $\mathcal{N}$ of $(0, T_*)$ in $X \times I$ such that all solutions of the equation $G(v, T)=0$ in $\mathcal{N}$ belong to the trivial solution line $\{(0, T)\}$ or to the curve \eqref{eq.B_Cap}. The intersection $(0, T_*)$ is called a bifurcation point.
\end{theorem}

We are now in a position to prove our main result.

\begin{proof}[Proof of \Cref{Th_main_B_Cap}]
Let
\begin{equation*}
X=Y=C_{e,m}^{2, \alpha}(\mathbb{R}/ \mathbb{Z}), \quad I=\mathbb{R}_+, \quad
U=\left\{w \in C_{e,m}^{2, \alpha}(\mathbb{R}/ \mathbb{Z}): \|w\|_{L^{\infty}(\mathbb{R}/ \mathbb{Z})}<1 \right\}.
\end{equation*}
We first note that by \Cref{Pr30}, $G: U\times I \to Y$ is a well-defined smooth map and for any $T \in I$,
\begin{equation*}
\begin{aligned}
G(0, T)
= \phi_{0,T}|_{\partial C^1_{1}}-\frac{1}{2}\int_{\partial C^{T}_{1}} \phi_{0,T}
= \Phi |_{\partial C^1_{1}}-\frac{1}{2}\int_{\partial C^{1}_{1}} \Phi
= 0.
\end{aligned}
\end{equation*}
Thus, condition (i) is satisfied.

\medskip

By \Cref{Pr40}, the kernel of the linearized operator $D_{v}G(0,T_{*})$ is one-dimensional and is spanned by the function $w(s)=\cos (2\pi s)$, i.e.,
\begin{equation*}
  \ker D_{v}G(0,T_{*})=\mathbb{R}\, w.
\end{equation*}
Since $D_{v}G(0,T_{*})$ is a Fredholm operator of index zero, it immediately follows that
\begin{equation*}
\mathrm{codim}\, \mathrm{Im}(D_{v}G(0,T_{*}))=1.
\end{equation*}
This verifies conditions (ii) and (iii).

\medskip

Finally, we verify the transversality condition, namely $D_{T}D_{v}G(0,T_{*})(w) \notin \mathrm{Im}(D_{v}G(0,T_{*})).$ For any $\xi \in \mathrm{Im}(D_{v}G(0,T_{*})) = \mathrm{Im}(H_{T_{*}})$, there exists some $v$ such that $\xi=H_{T_{*}}(v)$. Consequently,
\begin{equation*}
  \int_0^1 \xi w=\int_0^1 H_{T_{*}}(v)w =\int_0^1
	H_{T_{*}}(w)v=0,
\end{equation*}
where we used the self-adjointness of $H_{T_*}$ and the fact that $H_{T_{*}}(w)=0$. Since $H_T$ is a linear essentially self-adjoint operator with closed range, its image can be characterized as
\begin{equation*}
  \mathrm{Im}(H_{T_{*}})=\Bigg\{\xi:\int_0^1\xi w=0\Bigg\}.
\end{equation*}
Recalling that
\begin{equation*}
D_{T}D_{v}G(0,T_{*})(w) =\frac{d}{dT}\left(H_{T}(w)\right) \Big|_{T=T_{*}}.
\end{equation*}
Then, by \eqref{H_T}, we get
\begin{equation*}
  \frac{d}{dT}\left(H_{T}(w)\right) \Big|_{T=T_{*}}
  =\frac{d}{dT}\left(\hat \psi |_{\partial C_1^1} +\kappa w\right) \Big|_{T=T_{*}}
  =\frac{d}{dT} (\hat \psi|_{\partial C_1^1})  \Big|_{T=T_{*}}
  :=\dot{\hat{\psi}} |_{\partial C_1^1}.
\end{equation*}
Since $\hat \psi$ satisfies \eqref{psi_hat_2}, it follows that $\dot{\hat{\psi}}$ satisfies the following problem:
\begin{equation}\label{dot_psi_hat}
\begin{cases}
  \mathrm{div} (A\cdot \nabla \dot{\hat{\psi}})
  - b \dot{\hat{\psi}}
  =2T_*^{-3}\left(1+\Phi_y^2\right)^{-\frac{1}{2}} \hat{\psi}_{ss}(T_*)
  =:f(y,s) &~~\mbox{in}~~ C^1_{1},\\
  \dot{\hat{\psi}}_{\nu}=0 &~~\mbox{on}~~\partial C^1_{1},
\end{cases}
\end{equation}
To proceed, it suffices to show that
\begin{equation}\label{Targ}
  \int_0^1 \dot{\hat{\psi}}|_{\partial C_1^1} \cdot w\neq 0.
\end{equation}

In view of \eqref{psi_T_8}, we have
\begin{equation*}
  \hat \psi_{ss}(T_*)=-4\pi^2 \psi_1(y)\cos(2\pi s).
\end{equation*}
Consequently, the source term can be expressed as
\begin{equation*}
  f(y,s)=2T_*^{-3}\left(1+\Phi_y^2\right)^{-\frac{1}{2}} \hat{\psi}_{ss}
  =-8\pi^2 T_*^{-3}\left(1+\Phi_y^2\right)^{-\frac{1}{2}}
  \psi_1(y)\cos(2\pi s)
  =:g(y)\cos(2\pi s),
\end{equation*}
which is an even function with respect to $y$.

Furthermore, let us consider the associated one-dimensional problem:
\begin{equation}\label{tilde_psi}
\begin{cases}
  \left(A^{11} \tilde{\psi}_y\right)_y
  - \left(4\pi^2 A^{22}+b\right) \tilde{\psi}
  =g &~~\mbox{in}~~B,\\
  \tilde{\psi}_y=0 &~~\mbox{at}~~-1,1.
\end{cases}
\end{equation}
Note that $A^{11} > 0$ and $4\pi^2 A^{22}+b > 0$. By the Lax-Milgram theorem in $H^1(-1,1)$, the non-homogeneous Neumann boundary value problem \eqref{tilde_psi} possesses a unique even solution $\tilde{\psi}$ for any given sufficiently regular function $g$. We refer the reader to standard texts such as \cite[Chapter 8]{MR2759829} or \cite{MR2789179} for detailed theoretical treatments. By setting
\begin{equation*}
  \dot{\hat{\psi}}(y,s)=\tilde{\psi}(y) \cos(2\pi s),
\end{equation*}
one can easily verify that this separable form provides the unique solution to \eqref{dot_psi_hat}. Therefore,
\begin{equation*}
  \dot{\hat{\psi}}|_{\partial C_1^1}=\tilde{\psi}(1) w.
\end{equation*}

Next, in order to get \eqref{Targ},we only need to prove
\begin{equation*}
  \tilde{\psi}(1) \neq 0.
\end{equation*}
Indeed, since
\begin{equation*}
   \left(A^{11} \tilde{\psi}_y\right)_y
  - \left(4\pi^2 A^{22}+b\right) \tilde{\psi}
  =g<0
\end{equation*}
and $4\pi^2 A^{22}+b>0$, by the strong maximum principle \cite[Theorem 3.5]{MR1814364}, $\tilde{\psi}$ cannot achieve a non-positive minimum in $B$. Hence, if $\tilde{\psi}(1)=0$, we have that $\tilde{\psi}>0$ in $B$. Then, by Hopf's Lemma \cite[Lemma 3.4]{MR1814364}, $\tilde{\psi}_y(1) < 0$. This blatantly contradicts our homogeneous Neumann boundary condition $\tilde{\psi}_y(1) = 0$. Therefore, $\tilde{\psi}(1)\neq 0$.

From the above arguments, all the assumptions of \Cref{CR1} are satisfied. Hence, we obtain a $C^2$ curve $(v_{\eta},T_{\eta})$ (see \eqref{eq.B_Cap}), which satisfies (1) and (2) in \Cref{CR1}. For any $\eta \in(-\varepsilon,\varepsilon)$, there exists a unique solution $\phi_{\eta}$ of \eqref{capil_4} corresponding to $v=v_{\eta}$ and $T=T_{\eta}$. By $G\left(v_{\eta}, T_{\eta}\right)=0$, we conclude that $\phi_{\eta}$ is a constant (denoted by $c_{\eta}$) on $\partial C^1_1$ (see the definition of $G$ in \eqref{G}). Equivalently, there exists a unique solution $u_{\eta}$ of \eqref{capil_Neumann} with $u_{\eta}$ being the constant $c_{\eta}$ on $\partial C_{1+v_{\eta}}^{T_{\eta}}$. That is, $u_{\eta}$ is the unique solution of \eqref{capil_result}. Hence, the proof is completed.
\end{proof}

\medskip

\printbibliography

\end{document}